\documentclass[11pt]{article}

\usepackage{amssymb}
\usepackage{verbatim}
\usepackage{array}
\usepackage{latexsym}
\usepackage{enumerate}
\usepackage{float}
\usepackage{amsmath}
\usepackage{amsfonts}
\usepackage{amsthm}
\usepackage{mathrsfs}
\usepackage{color}
\usepackage[english]{babel}

\makeatletter
\usepackage{comment}
\let\wfs@comment@comment\comment
\let\comment\@undefined

\usepackage{changes}
\let\wfs@changes@comment\comment
\let\comment\@undefined

\newcommand\comment{%
    \ifthenelse{\equal{\@currenvir}{comment}}
    {\wfs@comment@comment}
    {\wfs@changes@comment}%
}

\definechangesauthor[name=Daniele, color=red]{DAN}
\definechangesauthor[name=Massimo, color=violet]{MAX}
\definechangesauthor[name=Giovanni, color=blue]{GIO}
\usepackage{lineno}

\newtheorem{theorem}{Theorem}
\newtheorem*{theorem*}{Main Theorem}
\newtheorem{proposition}[theorem]{Proposition}
\newtheorem{lemma}[theorem]{Lemma}

{\theoremstyle{definition}
\newtheorem*{definition*}{Definition}

\newtheorem{remark}[theorem]{Remark}

}

\newcommand{\fqn}{\mathbb{F}_{q^n}}
\newcommand{\fq}{\mathbb{F}_{q}}

\title{Towards the classification of exceptional scattered polynomials}
\author{Daniele Bartoli\thanks{Dipartimento di Matematica e Informatica, Universit\`a degli studi di Perugia,  Perugia, Italy. daniele.bartoli@unipg.it},
Massimo Giulietti\thanks{Dipartimento di Matematica e Informatica, Universit\`a degli studi di Perugia,  Perugia, Italy. massimo.giulietti@unipg.it}, and
Giovanni Zini\thanks{Dipartimento di Scienze Fisiche, Informatiche e Matematiche, Universit\`a degli Studi di Modena e Reggio Emilia, Modena, Italy.
giovanni.zini@unimore.it}.
}
\date{\today}

\begin{document}

\maketitle

\begin{abstract}
    Scattered polynomials over finite fields attracted an increasing attention in the last years. One of the reasons is their deep connection with Maximum Rank Distance (MRD) codes. Known classification results for exceptional scattered polynomials, i.e. polynomials which are scattered over infinite field extensions, are limited to the cases where their index $\ell$ is small, or a prime number larger than the $q$-degree $k$ of the polynomial, or an integer smaller than the $k$ in the case where $k$ is a prime. In this paper we completely classify exceptional scattered polynomials when the maximum between $\ell$ and $k$ is odd, and give partial results when it is even, extending a result of Ferraguti and Micheli in 2021.
\end{abstract}

\noindent {\bf MSC:} 11T06.\\
\noindent {\bf Keywords:} linearized polynomials, scattered polynomials, MRD codes, transitive finite linear groups.

\section{Introduction}

For $q$ a prime power and $n$ a positive integer, let $\mathbb{F}_{q^n}$ be the finite field with $q^n$ elements. A $q$-linearized polynomial $f(x)=\sum_{i=0}^{n-1} a_i x^{q^i}\in\mathbb{F}_{q^n}[x]$ is said to be \emph{scattered} of index $\ell\in\{0,\ldots,n-1\}$ over $\mathbb{F}_{q^n}$ if, for any $y,z\in\mathbb{F}_{q^n}^*$,
\begin{equation}\label{eq:scatt}
\frac{f(y)}{y^{q^\ell}}=\frac{f(z)}{z^{q^\ell}} \Longrightarrow \frac{y}{z}\in \mathbb{F}_{q};
\end{equation}
see \cite{MR3812212,MR3543528}. The $q$-degree of a linearized polynomial $f(x)=\sum_{i=0}^{n-1} a_i x^{q^i}$ is defined as $\max\{i : a_i\neq 0 \}$.  For a scattered polynomial $f(x)$ of index $\ell$ and $q$-degree $k$ we let $d:= \max\{k,\ell\}$.

Scattered polynomials $f(x)\in\mathbb{F}_{q^n}[x]$ are connected with \emph{scattered $\mathbb{F}_q$-subspaces} with respect to a Desarguesian spread. Recall that an \emph{$n$-spread} of the $r$-dimensional vector space $V$ over $\fqn$ is a set of $n$-dimensional $\fq$-subspaces of $V$ covering $V$ and pairwise intersecting trivially.
An $\fq$-subspace $U$ of $V$ is \emph{scattered} w.r.t. a spread $S$ if $U$ meets every element of $S$ in an $\fq$-space of dimension at most $1$.
A Deasarguesian spread of $V$ arises by applying a field reduction to the vectors of $V$; see \cite{MR3329986}. 
For a scattered polynomials $f(x)\in\mathbb{F}_{q^n}[x]$ of index $\ell$, the set 
\[
U_f=\{(x^{q^\ell},f(x))\colon x\in\mathbb{F}_{q^n}\}
\]
is a scattered $\mathbb{F}_q$-subspace of $\mathbb{F}_{q^n}\times\mathbb{F}_{q^n}$ w.r.t. a Desarguesian spread; see \cite{MR3812212,MR3543528}.

Scattered $\mathbb{F}_q$-subspaces  have  applications in different areas of mathematics, such as translation hyperovals \cite{MR1289081}, translation caps in affine spaces \cite{MR3800841}, two-intersection sets \cite{MR1772206}, blocking sets \cite{MR1790075}, translation spreads of the Cayley
generalized hexagon \cite{MR3403178}, finite semifields \cite{zbMATH06536325}, and graph theory \cite{MR818812}. Of particular interest is the connection between scattered subspaces and linear codes, namely MRD codes \cite{MR4143090,zini2021scattered,MR3543528}, whose relevance in communication theory relies on their applications to random linear network coding \cite{MR2450762} and cryptography \cite{gabidulin1991ideals}.

A scattered polynomial (of a certain index $\ell$) over $\mathbb{F}_{q^n}$ is said to be \emph{exceptional} if it is scattered  (with respect to the same index $\ell$) over infinitely many extensions $\mathbb{F}_{q^{nm}}$ of $\mathbb{F}_{q^n}$.

While several families of scattered polynomials have been constructed in recent years, only two families of exceptional polynomials are known so far and can be described as follows:
\begin{itemize}
    \item monomials of so-called pseudoregulus type,  $f(x)=x^{q^s}$ of index $0$, with $\gcd(s,n)=1$, see \cite{MR3199027};
    \item binomials of so-called LP type (named after Lunardon and Polverino who introduced them in \cite{MR1749427}), $f(x)=x+\delta x^{q^{2s}}$ of index $s$, with $\gcd(s,n)=1$ and $\mathrm{N}_{q^n/q}(\delta)\ne1$, see \cite{MR1749427}.
\end{itemize}

The natural problem of classifying exceptional scattered polynomials has been solved only for index $0$, $1$, and $2$ (see \cite{MR4190573,MR3812212}), and for prime values of $d$ (see \cite{MR4163074}).

In the investigation of exceptional scattered polynomials of index $\ell$, we can assume that a $q$-linearized polynomial $f(x)\in\mathbb{F}_{q^n}[x]$ is $\ell$-normalized, in the following sense (see \cite[Remark 4.1]{bartoli2022investigating}):
\begin{itemize}
    \item[(i)] the  $q$-degree $k$ of $f(x)$ is smaller than $n$;
    \item[(ii)] $f(x)$ is monic;
    \item[(iii)] the coefficient of $x^{q^\ell}$ in $f(x)$ is zero; 
    \item[(iv)] if $\ell>0$, then the coefficient of $x$ in $f(x)$ is nonzero, i.e. $f(x)$ is separable.
\end{itemize}

In this paper we completely classify scattered polynomials when $d=\max\{k,\ell\}$ is odd, combining a method introduced by Ferraguti and Micheli in \cite{MR4163074} with a classical result from Hering. A partial classification is also obtained when $d$ is even.
Our standpoint will be that of the theory of  Galois extensions of function fields. Our main result is the following. 

\begin{theorem*}
Let $f(x)\in\mathbb{F}_{q^n}[x]$  be an $\ell$-normalized $\mathbb{F}_q$-linearized polynomial of $q$-degree $k$, and let  $d:= \max\{k,\ell\}$. If $d$ is odd, then $f(x)$ is not exceptional scattered of index $\ell$ unless $f(x)$ is a monomial of pseudoregulus type.
\end{theorem*}

Main Theorem will be proved in Section \ref{sec:proof}. To this aim, we make use of the classification of transitive finite linear groups (Section \ref{sec:hering}), and their embeddings in higher dimensional linear groups (Section \ref{Sect:EmbeddingIssue}). Finally we give in Section \ref{sec:partial} some partial classification results also when $d$ is even.

\section{A result from Hering}\label{sec:hering}

Throughout the paper, $f(x)\in\mathbb{F}_{q^n}[x]$ is an $\ell$-normalized exceptional scattered polynomial of $q$-degree $k$, and $d=\max\{k,\ell\}$. Also, $\Gamma \mathrm{L}_{q}(a,q^b)$ denotes the subgroup of $\Gamma \mathrm{L}(a,q^b)$ defined as   $\Gamma \mathrm{L}_{q}(a,q^b)=\mathrm{GL}(a,q^b)\rtimes \mathrm{Aut}(\mathbb{F}_{q^b}:\mathbb{F}_q)$.

As usual, $\mathrm{Sp}(e,q)$, with $e$ even, denotes the symplectic group, i.e. the subgroup of ${\rm GL}(e,q)$ preserving a given non-degenerate alternating bilinear form. In terms of matrices, $\mathrm{Sp}(e,q)$ is made by the matrices $A\in{\rm GL}(e,q)$ such that $AHA^{\top}=H$, where $H$ is a given invertible skew-symmetric matrix.
For the rest of the paper, we fix
\begin{equation}\label{eq:symp}
H=\begin{pmatrix} 0 & I_{e/2} \\ -I_{e/2} & 0 \end{pmatrix}\in{\rm GL}(e,q),
\end{equation}
where $I_{e/2}$ denotes the identity matrix of size $e/2$.

Also, $G_2(q)$, with $q$ even, denotes the Cartan-Dickson-Chevalley exceptional group in its representation as a subgroup of $\mathrm{Sp}(6,q)$. For $q>2$, $G_2(q)$ is simple; for $q=2$, $G_2(2)$ contains the simple group $G_2(2)^{\prime}\cong{\rm PSU}(3,3)$ with index $2$.

The classical groups appearing in points 1. to 3. of Theorems \ref{th:hering} and \ref{th:hering_here} act on $V$ in their natural representations.

\begin{theorem}\label{th:hering}
Let $p$ be a prime, $q=p^a$, $d\in\mathbb{N}$, $V={\mathbb F}_q^d$, and $G$ be a subgroup of $\mathrm{GL}(d,q)$ acting transitively on $V\setminus\{\mathbf{0}\}$.
Then one of the following holds: 
\begin{enumerate}
    \item $\mathrm{SL}(e,q^{d/e})\triangleleft G \leq \Gamma \mathrm{L}_{q}(e,q^{d/e})$ for some $e\mid d$;
    \item $\mathrm{Sp}(e,q^{d/e})\triangleleft G \leq \Gamma \mathrm{L}_{q}(e,q^{d/e})$ for some even $e\mid d$ with $e\geq4$;
    \item $G_2(2^{d/6})^\prime \triangleleft G\leq \Gamma \mathrm{L}_{q}(6,2^{d/6})$, where $p=2$ and $6\mid d$;
    \item $q^d\in\{5^2,7^2,11^2,23^2,29^2,59^2,2^4,3^4,3^6\}$.
\end{enumerate}
\end{theorem}

\begin{proof}
The proof essentially goes back to Hering's papers  \cite{MR335659} and \cite{MR780488}, which build on his previous work \cite{MR224696}.
Up to our knowledge, in the literature Hering's results have been summarized as in Theorem \ref{th:hering}  only for $q=p$ a prime (for instance see \cite[Theorem 69.7]{MR2290291}), but in fact the case $q$ a prime power can be easily dealt with.

To see this, use the same notations of \cite[Section 5]{MR335659}. Let $L$ be a subset of ${\rm Hom}(V,V)$ maximal with respect to the following conditions: $L$ is normalized by $G$, $L$ contains the identity, and $L$ is a field with respect to the addition and multiplication in ${\rm Hom}(V,V)$.
Then $V$ is an $L$-vector space with scalar multiplication $\alpha v:= \alpha(v)$ for any  $\alpha\in L$ and $v\in V$.
By \cite[Lemma 5.2]{MR335659}, up to excluding a specific case (namely $n=2$, $p=3$, and $|L|=9$), $L$ is uniquely defined.

Clearly, $G$ normalizes the set $\mathcal{F}_q:=\{\tau_\lambda\colon \lambda\in\mathbb{F}_q\}$, where $\tau_\lambda$ is defined by $v\mapsto \lambda v$ for any $v\in V$, and $\mathcal{F}_q$ is a field isomorphic to $\mathbb{F}_q$, with the operations of ${\rm Hom}(V,V)$.
Therefore $L$ contains $\mathcal{F}_q$ and has size $q^{d/e}$ for some divisor $e$ of $d$. With the notation of \cite[Section 5]{MR335659}, we have $m=a\cdot\frac{d}{e}$ and $n^*=e$.

Then, as pointed out at the beginning of \cite[Section 5]{MR335659}, $G\leq \Gamma \mathrm{L}_{q}(e,q^{d/e})$, and the arguments of Hering's papers yield the claim. 
\end{proof}

\section{Embedding $\mathrm{\Gamma L}_q(n,q^m)$ in $\mathrm{GL}(nm,q)$}
\label{Sect:EmbeddingIssue}

For a positive integer $m$, consider the field extension $\mathbb F_{q^{m}}:\mathbb F_q$. Let $V$ be an $n$-dimensional vector space over $\mathbb F_{q^{m}}$. Then clearly $V$ is an $nm$-dimensional vector space over $\mathbb F_q$ and any $\mathbb F_{q^m}$-linear automorphism of $V$ is also 
an $\mathbb F_{q}$-linear automorphism. This provides the so-called {\em natural embedding} of the group of 
$\mathbb{F}_{q^m}$-linear automorphisms of $V$ into the group of  $\mathbb F_{q}$-linear automorphisms of $V$. 

If a basis $\mathcal A$ of $V$ over $\mathbb F_{q^m}$ and a basis $\mathcal C$ of $V$ over $\mathbb F_q$ are fixed, then clearly the natural embedding induces an embedding $\eta_{\mathcal A,\mathcal C}$ of $\mathrm{GL}(n,q^m)$ in $\mathrm{GL}(nm,q)$, which is again callled a natural embedding. Explicitly, for $T\in \mathrm{GL}(n,q^m)$, the matrix $\eta_{\mathcal A,\mathcal C}(T)$ acts on a vector $x$ in $\mathbb F_q^{nm}$ as follows:
first, consider the vector $v\in V$ such that $x=(v)_\mathcal C$ (i.e. the vector of coordinates of $v$ over the basis $\mathcal C$); then
let $y=(v)_{\mathcal A}\in \mathbb F_{q^m}^n$ 
and compute $z=Ty$;
let $w\in V$ be the vector  such that $z=(w)_\mathcal A$; finally, the image of $x$ by $\eta_{\mathcal A,\mathcal C}(T)$ is  $(w)_{\mathcal C}\in \mathbb F_{q}^{nm}$, that is 
\begin{equation}\label{eta_A,C}
    \eta_{\mathcal A,\mathcal C}(T)\cdot (v)_\mathcal C=(w)_{\mathcal C}.
\end{equation}

As far as the image of $\mathrm{GL}(n,q^m)$ in $\mathrm{GL}(nm,q)$ by a natural embedding is concerned, it is straightforward to check that different choices of bases produce conjugate subgroups of $\mathrm{GL}(nm,q)$.

If both $n$ and $m$ are odd, then any embedding of $\mathrm{GL}(n,q^m)$ in $\mathrm{GL}(nm,q)$ is actually natural (or, equivalently, all subgroups of $\mathrm{GL}(nm,q)$ which are isomorphic to $\mathrm{GL}(n,q^m)$ are conjugate). As we could not find a reference for this fact, we deduce it from  a result by Kantor.

\begin{proposition}\label{Prop:EmbeddingGL}
Any embedding of $\mathrm{GL}(n,q^m)$, $nm>1$, $n,m$ odd, in $\mathrm{GL}(nm,q)$ is natural.
\end{proposition}
\begin{proof}
Let $\eta^{\prime}$ be any   embedding of $\mathrm{GL}(n,q^m)$ in $\mathrm{GL}(nm,q)$.  The group $G=\eta^{\prime}(\mathrm{GL}(n,q^m))$ contains a Singer cycle of $\mathrm{GL}(nm,q)$ (of order $q^{nm}-1$). By \cite{MR561126}, $G\trianglerighteq \mathrm{GL}(nm/s,q^s)$, embedded naturally for some divisor $s$ of $nm$. Clearly $s\geq m$, otherwise $|G|<|\mathrm{GL}(nm/s,q^s)|$.
\begin{itemize}
    \item Suppose $s=nm$. Then $\mathrm{GL}(1,q^{nm})$ is a normal subgroup of $G$. Hence $G$ is contained in the normalizer (in $\mathrm{GL}(nm,q)$) of $\mathrm{GL}(1,q^{nm})$, which by \cite[Sect.II.7]{MR0224703} is equal to $\mathrm{GL}(1,q^{nm})\rtimes \mathrm{Aut}(\mathbb{F}_q^{nm})$. But $|G| > | \mathrm{GL}(1,q^{nm})\rtimes \mathrm{Aut}(\mathbb{F}_q^{nm})|$, a contradiction. 
    \item Suppose that $m<s<nm$. Since, by assumption, both $n$ and $m$ are odd, $nm\neq 2s$. Following the proof of \cite[page 232]{MR561126}, $G < \Gamma L (nm/s,q^s)$. Thus $|\mathrm{GL}(n,q^m)|=|G| < |\Gamma L (nm/s,q^s)|$, a contradiction. 
\end{itemize}
Therefore $s=m$ and the embedding $\eta^{\prime}$ is natural.
\end{proof}

For the rest of the paper it is convenient to write explicitly a (natural) embedding of $\mathrm{GL}(n,q^m)$ into $\mathrm{GL}(nm,q)$.

Let $\gamma$ be a primitive element of $\mathbb F_{q^m}$, so that $\mathcal{B}=\{1,\gamma,\ldots, \gamma^{m-1}\}$  is an $\mathbb{F}_q$-basis  of $\mathbb{F}_{q^m}$.  Consider also an $\mathbb{F}_{q^m}$-basis $\mathcal{A}=\{\alpha_1,\ldots, \alpha_n\}$ of $\mathbb{F}_{q^{nm}}$. Clearly, the set  \begin{equation}\label{Eq:C}
    \mathcal{C}=\{\alpha_1,\alpha_1\gamma\ldots,\alpha_1\gamma^{m-1},\ldots,\alpha_n,\alpha_n\gamma,\ldots,\alpha_n\gamma^{m-1}\}
\end{equation} is an $\mathbb F_q$-basis of $\mathbb F_{q^{nm}}$.
We are going to describe $\eta_{\mathcal{A},\mathcal{C}}$ explicitly, as defined in \eqref{eta_A,C}.

Consider the map  $\varphi: \mathbb{F}_{q^m}\to \mathbb{F}_q^{m\times m}$ which maps 
$a\in\mathbb{F}_{q^m}$ to the matrix whose $(i+1)$-th column consists of the components of $\gamma^i\cdot a$ over the basis $\mathcal B$, for $i=0,\ldots, m-1$. In more explicit terms,
if $C$ is the $m\times m$ matrix over $\mathbb F_q$ representing the $\mathbb F_q$-linear map of $\mathbb F_{q^m}$ $x\mapsto \gamma\cdot x$ with respect to $\mathcal B$, then 
\begin{equation}\label{Def:varphi}
\varphi(a)=\left(
\begin{array}{c|c|c|c}
&&&\\
 (a)_{\mathcal{B}}&C\cdot (a)_{\mathcal{B}}&\cdots&C^{m-1}\cdot (a)_{\mathcal{B}}\\
&&&
\end{array}
\right).
\end{equation}
It is well-known that $C$ is the companion matrix of the minimal polynomial of $\gamma$ over $\mathbb F_q$. Now, if $T=(a_{i,j})_{i,j=1,\ldots,n}\in \mathrm{GL}(n,q^m)$ then it is straightforward to check that $\eta_{\mathcal{A},\mathcal{C}}(T)=(\varphi(a_{i,j}))_{i,j=1,\ldots,n}\in GL(nm,q)$.

Let $nm$ be odd and consider an embedding $\zeta$ of $\Gamma \mathrm{L}_q(n,q^m)=\mathrm{GL}(n,q^m)\rtimes \mathrm{Aut}(\mathbb{F}_{q^m}: \mathbb{F}_q)$ in $\mathrm{GL}(nm,q)$. By Proposition \ref{Prop:EmbeddingGL}, $\zeta_{|\mathrm{GL}(n,q^m)}$ is a natural embedding. Therefore, there exists an inner automorphism $\beta$
of $GL(nm,q)$ such that $(\beta \circ \zeta)(T)=(\varphi(a_{i,j}))_{i,j=1,\ldots,n}$ for  
$T=(a_{i,j})_{i,j=1,\ldots,n}\in \mathrm{GL}(n,q^m)$ with $\varphi$ as in \eqref{Def:varphi}.

Let $M=(M_{i,j})_{i,j =1,\ldots,n}\in \mathrm{GL}(nm,q)$ be the image $(\beta \circ \zeta)(\phi)$ of the Frobenius map $\phi:a\mapsto a^q$  in $\mathrm{Aut}(\mathbb{F}_{q^m}: \mathbb{F}_q)$; here, $M_{i,j}$ is an $m\times m$ matrix over $\mathbb{F}_q$ for any $i,j=1,\ldots,n$. 
For any $a\in \mathbb{F}_{q^m}$ we have
$$diag(a,1,\ldots,1)\circ \phi=\phi \circ diag(a^{q^{m-1}},1,\ldots,1), $$
and hence
$$(\beta\circ \zeta)(diag(a,1,\ldots,1)\circ \phi)=(\beta\circ \zeta)(\phi \circ diag(a^{q^{m-1}},1,\ldots,1)).$$
Therefore
$$diag(\varphi(a),\varphi(1),\ldots,\varphi(1))\cdot M =M\cdot diag(\varphi(a^{q^{m-1}}),\varphi(1),\ldots,\varphi(1)), $$
that is 
{\scriptsize $$\left(\begin{array}{cccc}
\varphi(a)\cdot M_{1,1}&\varphi(a)\cdot M_{1,2}&\cdots & \varphi(a)\cdot M_{1,n}\\
\varphi(1)\cdot M_{2,1}&\varphi(1)\cdot M_{2,2}&\cdots & \varphi(1)\cdot M_{2,n}\\
\vdots&\vdots& & \vdots\\
\varphi(1)\cdot M_{n,1}&\varphi(1)\cdot M_{n,2}&\cdots & \varphi(1)\cdot M_{n,n}\\
\end{array}\right)
=
\left(\begin{array}{cccc}
 M_{1,1}\cdot \varphi(a^{q^{m-1}}) & M_{1,2}\cdot \varphi(1) &\cdots &  M_{1,n}\cdot \varphi(1)\\
 M_{2,1}\cdot \varphi(a^{q^{m-1}})& M_{2,2}\cdot \varphi(1)&\cdots &  M_{2,n}\cdot \varphi(1)\\
\vdots&\vdots& & \vdots\\
 M_{n,1}\cdot \varphi(a^{q^{m-1}})& M_{n,2}\cdot \varphi(1)&\cdots &  M_{n,n}\cdot \varphi(1)\\
\end{array}\right).
$$}

In particular, for each $i\neq 1$ and each $a\in \mathbb{F}_{q^m}^*$, $M_{1,i}\cdot \varphi(1)=\varphi(a)\cdot M_{1,i}$. Note that $\det(\varphi(a)-\varphi(b))\neq 0$ for each $a\neq b\in \mathbb{F}_{q^m}^*.$ Therefore $M_{1,i}=\overline{O}$ (the zero matrix) for $i=2,\ldots,n$. The same argument applies to each $M_{i,j}$ with $i\neq j$ and thus \begin{equation}\label{Eq:M}
    M=diag (\overline{M},\ldots,\overline{M}),
\end{equation} where $\overline{M}$ is provided by the (unique) embedding of $\Gamma \mathrm{L}_q(1,q^{m})$ into $\mathrm{GL}(m,q)$ associated with the fixed $\mathbb{F}_q$-basis $\mathcal{B}$ of $\mathbb{F}_{q^m}$; see  \cite[Sect.II.7]{MR0224703}.

Observe that $\Gamma \mathrm{L}_q(n,q^m)$ contains a unique subgroup isomorphic to $\mathrm{SL}(n,q^m)$; this can be shown by noting that, for any pair $(n,q^m)\ne(2,2),(2,3)$, the subgroup $\mathrm{SL}(n,q^m)$ is the last term $\Gamma \mathrm{L}_q(n,q^m)^{(\infty)}$ of the commutator series of $\Gamma \mathrm{L}_q(n,q^m)$.

\section{Proof of Main Theorem}\label{sec:proof}
In the rest of the paper, the notation on Galois extensions of global function fields is as in \cite[Section 2]{MR4163074}; see also \cite{MR4053386}.
For an $\ell$-normalized $\mathbb{F}_q$-linearized polynomial $f(x)\in\mathbb{F}_{q^n}[x]$, let $s$ be a transcendental over $\mathbb{F}_{q^n}$ and $S$ be the splitting field of $f(x)-sx^{q^{\ell}}\in\mathbb{F}_{q^n}(s)[x]$ over $\mathbb{F}_{q^n}(s)$.
For any positive integer $r$, denote by $S_r$ the compositum function field $S\cdot \mathbb{F}_{q^{nr}}$, by $k_r$ the field of constants of $S_r$, by $G_r^{\rm arith}$ and $G_r^{\rm geom}$ the arithmetic and the geometric Galois group of $S_r :\mathbb{F}_{q^{nr}}(s)$, and by $\varphi_r$ the isomorphism $G_r^{\rm arith}/G_r^{\rm geom}\to{\rm Gal}(k_r :\mathbb{F}_{q^{nr}})$.

As a direct application of Theorem \ref{th:hering} we have the following.

\begin{theorem}\label{th:hering_here}
Let $f(x)\in\mathbb{F}_{q^n}[x]$ be an $\ell$-normalized $\mathbb{F}_q$-linearized polynomial of $q$-degree $k<n$, and let $d:= \max\{k,\ell\}$. Then $f(x)$ is not exceptional scattered unless one of the following holds:
\begin{enumerate}
    \item $\mathrm{SL}(e,q^{d/e})\triangleleft G_r^{\rm geom} \triangleleft G_r^{\rm arith} \leq \Gamma \mathrm{L}_{q}(e,q^{d/e})$ for some $e\mid d$;
    \item $\mathrm{Sp}(e,q^{d/e})\triangleleft G_r^{\rm geom} \triangleleft G_r^{\rm arith} \leq \Gamma \mathrm{L}_{q}(e,q^{d/e})$, for some even $e\mid d$ with $e\geq4$;
    \item $G_2(2^{d/6})^\prime \triangleleft G_r^{\rm geom}\triangleleft G_r^{\rm arith}\leq \Gamma \mathrm{L}_{q}(6,2^{d/6})$, where $p=2$ and $6\mid d$;
    \item $q^d\in\{5^2,7^2,11^2,23^2,29^2,59^2,2^4,3^4,3^6\}$.
\end{enumerate}
\end{theorem}

By \cite[Lemma 2.2]{MR4163074} there exists a constant $C>0$ depending on $S :\mathbb{F}_{q^n}(s)$ such that for any $r$ satisfying $q^{nr}>C$ the following property holds: every $\gamma\in G_r^{\rm arith}$ such that $\varphi_r(\gamma)$ is the Frobenius automorphism for the extension $k_r:\mathbb{F}_{q^{nr}}$ is also a Frobenius at an unramified place at finite of degree $1$ of $\mathbb{F}_{q^{nr}}(s)$.
For the rest of the paper, $r$ is assumed to satisfy $q^{nr}>C$. Under this assumption, $G_r^{\rm arith}\ne G_{r}^{\rm geom}$ holds; see \cite[Corollary 2.8]{MR4163074}.

\begin{remark}\label{Remark}
Let $\eta$ be any embedding of $\Gamma \mathrm{L}_q(e,q^{d/e})=\mathrm{GL}(e,q^{d/e})\rtimes \textrm{Aut}(\mathbb{F}_{q^{d/e}})=\mathrm{GL}(e,q^{d/e})\rtimes\langle \phi\rangle$ in $\mathrm{GL}(d,q)$. By Section \ref{Sect:EmbeddingIssue}, we can assume up to conjugation that 
$\eta_{|\mathrm{GL}(e,q^{d/e})}=\eta_{\mathcal{A},\mathcal{C}}$ for some $\mathbb{F}_{q^{d/e}}$-basis $\mathcal{A}$ of $\mathbb{F}_{q^{d}}$ and $\mathcal{C}$ as in \eqref{Eq:C}, such that $\eta(\phi)=M$ is as in \eqref{Eq:M}.
\end{remark}

\begin{proposition}\label{prop:SL}
Suppose that $d>2$ and that there exists $e\mid d$ such that $e>2$ and
$\mathrm{SL}(e,q^{d/e})\triangleleft G_r^{\rm geom}\triangleleft G_r^{\rm arith} \leq {\rm \Gamma L}_{q}(e,q^{d/e})$.
Then for any $\gamma \in G_r^{\rm arith}$
 there exists $\alpha\in G_r^{\rm geom} $ with  $ {\rm rank}(\eta(\alpha  \gamma)-\mathbb{I}_{d})<d-1$.
\end{proposition}

\begin{proof}
By Remark \ref{Remark}, it is enough to prove the claim for the case $\eta_{|\mathrm{GL}(e,q^{d/e})}=\eta_{\mathcal{A},\mathcal{C}}$ and $\eta(\phi)=M$, since the existence of $\alpha\in G_r^{\rm geom}$ with ${\rm rank}(\eta(\alpha  \gamma)-\mathbb{I}_{d})<d-1$ is invariant under conjugation in ${\rm GL}(d,q)$.

Write $\gamma =\beta\phi^{j} \in G_r^{\rm arith}$ with $\beta\in{\rm GL}(e,q^{d/e})$ and $j\in\{1,\ldots,e\}$, and write $B=\eta(\beta)$, so that $\eta(\gamma)=B\cdot M^{j}\in \mathrm{GL}(d,q)$. 
 We aim to determine $A=\eta(\alpha)\in \eta(\mathrm{SL}(e,q^{d/e}))$ such that ${\rm rank}(A\cdot B\cdot M^j-\mathbb{I}_{d})<d-1$. 
 Since $A\cdot B\cdot M^j=M^{j}\cdot \overline{A}\cdot \overline{B}$ for some $\overline{A}\in \eta(\mathrm{SL}(e,q^{d/e}))$,  $\overline{B}\in \eta(\mathrm{GL}(e,q^{d/e}))$, it is enough to find $\overline{A}\in \eta(\mathrm{SL}(e,q^{d/e}))$ such that ${\rm rank}(M^{j}\cdot \overline{A}\cdot \overline{B}-\mathbb{I}_{d})<d-1$. 
 
 Let $(x_{1}^{(1)},\ldots,x_{d/e}^{(1)}, \ldots, x_{1}^{(e)},\ldots,x_{d/e}^{(e)})$ and $(\overline{x}_{1}^{(1)},\ldots,\overline{x}_{d/e}^{(1)}, \ldots, \overline{x}_{1}^{(e)},\ldots,\overline{x}_{d/e}^{(e)})$ be respectively the first and the $(d/e+1)$-th column of $M^{e-j}=(M^j)^{-1}$ and let $y^{(i)}\in \mathbb{F}_{q^{d/e}}$, $i=1,\ldots,e$, and $\overline{y}^{(i)}\in \mathbb{F}_{q^{d/e}}$, $i=1,\ldots,e$, be such that $y^{(i)}_{\mathcal{B}}=(C^{i-1})^{-1}(x_{1}^{(i)},\ldots,x_{d/e}^{(i)})$ and $\overline{y}^{(i)}_{\mathcal{B}}=(C^{i-1})^{-1}(\overline{x}_{1}^{(i)},\ldots,\overline{x}_{d/e}^{(i)})$. 
 
 Consider a matrix $D\in \mathrm{SL}(e,q^{d/e})$ whose first two columns are 
    \begin{equation}\label{eq:trovainGL}
    (y^{(1)},y^{(2)}, \ldots, y^{(e)})=(y^{(1)},0,\ldots,0) \quad \textrm{ and }\quad  (\overline{y}^{(1)},\overline{y}^{(2)}, \ldots, \overline{y}^{(e)})=(0,\overline{y}^{(2)},0,\ldots,0)\end{equation}
 and such that $\det(D)=\det(\eta^{-1}(\overline{B}))$; such a matrix $D$ exists, because $e>2$.
 
 Now, consider the matrix $\overline{A}:=\eta(D)\cdot \overline{B}^{-1}\in\eta(\mathrm{SL}(e,q^{d/e}))$. It is readily seen that the first and the $(d/e+1)$-th column of 
 $M^{j}\cdot \overline{A}\cdot \overline{B}$ are  $$(1,0,\ldots,0) \quad   \textrm{ and } \quad  (0,0,\ldots,0,\underbrace{1}_{d/e+1},0,\ldots,0),$$ respectively.
 This shows that ${\rm rank}(M^{j}\cdot \overline{A}\cdot \overline{B}-\mathbb{I}_{d})\leq d-2$ and the claim follows.
\end{proof}

\begin{theorem}\label{th:main}
Let $f(x)\in\mathbb{F}_{q^n}[x]$ be an $\ell$-normalized $\mathbb{F}_q$-linearized polynomial of $q$-degree $k<n$, and let $d:= \max\{k,\ell\}$. Then $f(x)$ is not exceptional scattered unless one of the following holds:
\begin{itemize}
    \item $f(x)$ is a monomial of pseudoregulus type; or
    \item $G_r^{\rm arith}\not\leq \mathrm{\Gamma L}_q(1,q^{d})$ and $\mathrm{SL}(e,q^{d/e})\not\leq G_r^{\rm geom}$ for any divisor $e>2$ of $d$.
\end{itemize}
\end{theorem}

\begin{proof}
The case $q^d=9^3$ follows by \cite[Theorem 1]{MR4163074}.
Suppose that $f(x)$ is exceptional scattered and not of pseudoregulus type.
\begin{itemize}
\item Suppose that $\mathrm{SL}(e,q^{d/e})\leq G_r^{\rm geom}$ for some divisor $e>2$ of $d$.
By Theorem \ref{th:hering_here},
\[
\mathrm{SL}(e,q^{d/e})\triangleleft G_r^{\rm geom}\triangleleft G_r^{\rm arith}\leq \Gamma L_{q}(e,q^{d/e}).
\]
Then by Proposition \ref{prop:SL} there exist $\gamma\in G_r^{\rm arith}$ and $\alpha\in G_r^{\rm geom}$ such that ${\rm rank}(\eta(\alpha\gamma)-\mathbb{I}_d)<d-1$.
By \cite[Theorem 2.7]{MR4163074}, $f(x)$ is not exceptional scattered, a contradiction.
\item Suppose that 
\[
G_r^{\rm geom}\triangleleft G_r^{\rm arith} \leq \Gamma L_{q}(1,q^{d}).
\]
Now we argue as in \cite[Section 4]{MR4163074}. Since $f(x)-sx^{q^\ell}\in\mathbb{F}_{q^n}(s)[x]$ has exactly $q^d-1$ non-zero roots, the transitivity of $G_r^{\rm geom}$ on such roots implies $(q^d-1)\mid |G_r^{\rm geom}|$.
Thus, $|G_r^{\rm geom}|=i(q^d-1)$ and $|G_r^{\rm arith}|=j(q^d-1)$ with $1\leq i\mid j\leq d$.
As in \cite[Proof of Theorem 1.4]{MR4163074}, one gets that $|q^\ell-q^k|$ divides $i(q^d-1)$.

Suppose $\ell<k=d$. Then
\begin{equation}\label{eq:cond1}
\frac{q^k-q^\ell}{q^{\gcd(k,\ell)}-1}\,\Big|\, i<r.
\end{equation}
By considering separately $\ell\mid k$ or $\ell\nmid k$, one gets $k>q^{k/2}-1$, and hence, by direct computations, a contradiction to $\eqref{eq:cond1}$.

Then $k<\ell=d$ and
\[
\frac{q^\ell-q^k}{q^{\gcd(k,\ell)}-1}\,\Big|\, i<\ell.
\]
If $k>0$, then $\ell>q^{\ell/2}-1$ with $\ell>1$, whence a contradiction as above.
Then $k=0$, $(x^{q^\ell},f(x))=(x^{q^\ell},x)$, and thus $f(x)$ is a monomial, a contradiction.
\end{itemize}
\end{proof}

It is readily seen that Theorems \ref{th:hering_here} and \ref{th:main} prove Main Theorem.
They will be useful in the next section to provide partial results also when $d$ is even.

\section{Partial results for $d$ even}\label{sec:partial}

The notation in this section is the same as in Section \ref{sec:proof}. The following lemma is a simplified version of a well-known result by van der Waerden \cite{MR1513024}; see also \cite[Lemma 2.4]{MR4163074}.

\begin{lemma}\label{lem:van}
Let $f(x)\in\mathbb{F}_{q^n}[x]$ be an $\ell$-normalized $\mathbb{F}_q$-linearized polynomial of $q$-degree $k<n$, and $d:= \max\{k,\ell\}$.
Let $V$ be the set of $q^d$ roots of $f(x)-sx^{q^{\ell}}$, and $z\in V\setminus\{0\}$.
Let $P$ be a place of $k_r(s)$ and $\mathcal{Q}$ be the set of places of $k_r(s,z)$ lying over $P$.

Then $G_r^{\rm geom}$ has a subgroup $H$ whose set $\mathcal{S}$ of orbits on $V\setminus\{0\}$ is in bijection with $\mathcal{Q}$. If $Q\in\mathcal{Q}$ and $\mathcal{O}\in\mathcal{S}$ correspond each other in this bijection, then $e(Q|P)=|\mathcal{O|}$. 
\end{lemma}

We are now able to give some additional information on Case 1. in Theorem \ref{th:hering_here}.

\begin{proposition}\label{prop:SL2}
Let $f(x)\in\mathbb{F}_{q^n}[x]$ be an $\ell$-normalized $\mathbb{F}_q$-linearized polynomial of $q$-degree $k<n$, and let $d:= \max\{k,\ell\}$.
Suppose that $f(x)$ is not a monomial. If there is a divisor $e$ of $d$ such that
\[\mathrm{SL}(e,q^{d/e})\triangleleft G_r^{\rm geom} \triangleleft G_r^{\rm arith} \leq \Gamma \mathrm{L}_{q}(e,q^{d/e}),\]
then $e=2$, $d=k$ and $\ell=k/2$.
\end{proposition}

\begin{proof}
    The claim $e=2$ follows from Theorem \ref{th:main}.
    
    The groups $G_r^{\rm geom}$ and $G_r^{\rm arith}$ are contained in the subgroup $\mathcal{G}:=\eta({\rm \Gamma L}_q(2,q^{d/2}))$ of $ {\rm GL}(V)$, and the results of Section \ref{Sect:EmbeddingIssue} prove the following fact: in its action on $V= {\rm GL}(d,q)$, $\mathcal{G}$ preserves a $\frac{d}{2}$-spread, namely the Desarguesian $\frac{d}{2}$-spread $\mathcal{D}$ associated with the natural embedding $\eta$; see also \cite{MR561126}. To be more precise, with the notation of \cite[Section 5]{MR335659} that we recalled in the proof of Theorem \ref{th:hering}: $V$ is a $2$-dimensional vector space over a finite field $L\cong\mathbb{F}_{q^{d/2}}$, and every spread element consists of the $L$-multiples of a given vector of $V$.
    This implies in particular that any subgroup $H$ of $G_r^{\rm geom}$ preserves the partition $\mathcal{D}$ of $V\setminus\{0\}$ into $(q^{d/2}-1)$-subsets.
    
    Suppose by contradiction that $d=\ell$, and let $z\in V\setminus\{0\}$ be such that $s=f(z)/z^{q^\ell}$. Then the zero of $s$ in $k_r(s)$ lies exactly under $q^k$ places of $k_r(s,z)$: the pole of $z$, with ramification index $q^{\ell}-q^{k}$, and the $q^k-1$ distinct zeros of $f(z)/z$, each with ramification index $1$. By Lemma \ref{lem:van}, there exists $H\leq G_r^{\rm geom}$ with $q^k$ fixed vectors $v_0=0,\ldots,v_{q^k}\in V$ and a single remaining orbit of length $q^{d}-q^k$. Since $H$ preserves $\mathcal{D}$, we have that $\{v_1,\ldots,v_{q^k}\}$ is a single spread element $\mathcal{E}=\{\lambda v\mid \lambda \in L\}\in\mathcal{D}$ and $(q^{d/2}-1)\mid(q^k-1)$, whence $k=d/2$.
    Also, as $H$ fixes $\lambda v$ for any $\lambda\in L$, it follows that $H$ is not only $L$-semilinear, but rather $L$-linear. Thus, $H$ is contained in ${\rm GL}(2,q^{d/2})$.
    Since $H$ fixes a non-zero vector $v\in V$ and has an orbit of length $q^{d/2}(q^{d/2}-1)$, $H$ contains the full stabilizer ${\rm GL}(2,q^{d/2})_v$ of $v$ in ${\rm GL}(2,q^{d/2})$. Therefore $G_r^{\rm geom}$ contains both ${\rm GL}(2,q^{d/2})_v$ and ${\rm SL}(2,q^{d/2})$, whence ${\rm GL}(2,q^{d/2})\subseteq G_r^{\rm geom}$.
    
    Let $\gamma\in G_r^{\rm arith}$. Arguing as in the proof of Proposition \ref{prop:SL}, it is possible to find $\alpha\in G_r^{\rm geom}$ such that ${\rm rank}(\eta(\alpha\gamma)-\mathbb{I}_d)<d-1$. In fact, a matrix $D$ with the two columns as in Equation \eqref{eq:trovainGL} can always be found in $G_r^{\rm geom}$, because we have proved that $G_r^{\rm geom}$ contains not only ${\rm SL}(2,q^{d/2})$ but also ${\rm GL}(2,q^{d/2})$. Thus, a contradiction is obtained in the same way as in the proof of Proposition \ref{prop:SL}.
    
    We have then shown that $d=k$. Define $z$ as above and let $P$ be the pole of $s$ in $k_r(z)$. There are two places over $P$ in $k_r(s,z)$, namely the zero and the pole of $z$, with ramification index $q^{\ell}-1$ and $q^{d}-q^{\ell}$, respectively. By Lemma \ref{lem:van}, $q^{\ell}-1$ and $q^{d}-q^{\ell}$ are the sizes of the orbits on $V\setminus\{0\}$ of a subgroup $H\leq G_r^{\rm geom}$. As $H$ preserves $\mathcal{D}$, we have $\ell=d/2$. The claim is proved.
\end{proof}

We can now provide a partial classification also for the $d$ even case, as stated in Theorem \ref{d_even}.

\begin{theorem}\label{d_even}
Let $f(x)\in\mathbb{F}_{q^n}[x]$ be an $\ell$-normalized $\mathbb{F}_q$-linearized polynomial of $q$-degree $k<n$, and let $d:= \max\{k,\ell\}$.
Suppose that $f(x)$ is not a monomial.
Then $f(x)$ is not exceptional scattered unless one of the following holds:
\begin{enumerate}
    \item $\mathrm{SL}(2,q^{d/2})\triangleleft G_r^{\rm geom} \triangleleft G_r^{\rm arith} \leq \Gamma \mathrm{L}_{q}(2,q^{d/2})$;
    \item $G_2(2^{d/6})^\prime \triangleleft G_r^{\rm geom}\triangleleft G_r^{\rm arith}\leq \Gamma \mathrm{L}_{q}(6,2^{d/6})$, where $p=2$ and $6\mid d$;
    \item $q^d\in\{5^2,7^2,11^2,23^2,29^2,59^2,2^4,3^4,3^6\}$.
\end{enumerate}
\end{theorem}

\begin{proof}
    By Theorem \ref{th:main} and Proposition \ref{prop:SL2}, it is enough to show by contradiction that Case 2. in Theorem \ref{th:hering_here} cannot occur.
    
    We first recall the automorphism group ${\rm Aut}(\mathrm{Sp}(e,q^{d/e}))$ of ${\rm Sp}(e,q^{d/e})$, as a subgroup of ${\rm \Gamma L}(e,q^{d/e})$. Indeed, by the main theorem in \cite{HUA}, ${\rm Aut}(\mathrm{Sp}(e,q^{d/e}))$ is made by the maps
    \begin{equation}
    \begin{array}{ccc}
    {\rm Sp}(e,q^{d/e})&\to& {\rm Sp}(e,q^{d/e})\\
    A&\mapsto& A^{\prime}\circ R_a\circ  \sigma \circ A\circ R_a^{-1} \circ(A^{\prime})^{-1},
    \end{array}
    \end{equation}
    where $A^{\prime}$ runs in ${\rm Sp}(e,q^{d/e})$; $R_a=\begin{pmatrix}\mathbb{I}_{e/2}&{\bf 0} \\ {\bf 0} & a\mathbb{I}_{e/2}  \end{pmatrix}$, where $a$ is either $1$ or a fixed nonsquare of $\mathbb{F}_{q^{d/e}}$; and $\sigma\in{\rm Aut}(\mathbb{F}_{q^{d/e}})$.
    Define $\mathcal{R}_a=\{R_{\sigma(a)}\colon \sigma\in{\rm Aut}(\mathbb{F}_{q^{d/e}})\}$ and  $(\mathcal{R}_a)_q=\{R_{\sigma(a)}\colon \sigma\in{\rm Aut}(\mathbb{F}_{q^{d/e}}: \mathbb{F}_q)\}$.
    Since $R_a \circ \sigma= \sigma \circ R_{\sigma^{-1}(a)}$, we have $\mathcal{R}_a\leq{\rm Aut}(\mathbb{F}_{q^{d/e}})$ and hence ${\rm Aut}({\rm Sp}(e,q^{d/e}))$ can be identified with
    \[
    \mathcal{A}:= ({\rm Sp}(e,q^{d/e})\rtimes \mathcal{R}_a) \rtimes  {\rm Aut}(\mathbb{F}_{q^{d/e}})\leq {\rm \Gamma L}(e,q^{d/e}).
    \]
    
    Therefore, the intersection ${\rm Aut}_q(\mathrm{Sp}(e,q^{d/e}))$ between ${\rm Aut}({\rm Sp}(e,q^{d/e}))$ and ${\rm \Gamma L}_q(e,q^{d/e})$ equals
    \[
    \mathcal{A}_q:=({\rm Sp}(e,q^{d/e})\rtimes (\mathcal{R}_a)_q) \rtimes  {\rm Aut}(\mathbb{F}_{q^{d/e}}:\mathbb{F}_q)\leq {\rm \Gamma L}_q(e,q^{d/e}).
    \]
    Also, the centralizer $\mathcal{Z}_q$ of ${\rm Sp}(e,q^{d/e})$ in ${\rm \Gamma L}_q(e,q^{d/e})$ is $\mathcal{Z}_q=\{S_\alpha : \alpha \in \mathbb{F}_{q^{d/e}}\}$, where $S_\alpha$ is the scalar matrix with $\alpha$ on the diagonal entries; see e.g. \cite[4.2.5]{OMEARA}.
    
    Since $G_r^{\rm arith}$ is contained in ${\rm \Gamma L}_q(e,q^{d/e})$ and contains ${\rm Sp}(e,q^{d/e})$ as a normal subgroup,
    \[
    G_r^{\rm arith}\leq\langle \mathcal{A}_q, \mathcal{Z}_q \rangle= \mathcal{Z}_q \times \mathcal{A}_q
    \]
    and hence
    \begin{equation}\label{eq:aritto}
    G_r^{\rm arith}\leq\left\{ R\circ\sigma \circ D\circ A \quad:\quad  R\in(\mathcal{R}_a)_q, \ \sigma \in {\rm Aut}(\mathbb{F}_{q^{d/e}}:\mathbb{F}_q), \ S \in \mathcal{Z}_q, \ A\in{\rm Sp}(e,q^{d/e})\right\}.
    \end{equation}
    In analogy with the proof of Proposition \ref{prop:SL}, we will now show that, for any $M\in G_r^{\rm arith}$, there exists $B\in G_r^{\rm geom}$ such that ${\rm rank}(\eta(MB)-\mathbb{I}_d)<d-1$.
    This is enough to conclude the proof: in fact, from $G_r^{\rm geom}\triangleleft G_r^{\rm arith}$ it follows that there exists $B^\prime\in G_r^{\rm geom}$ such that ${\rm rank}(\eta(MB^{\prime})-\mathbb{I}_d)<d-1$, which yields a contradiction to $f(x)$ being exceptional scattered by \cite[Theorem 2.7]{MR4163074}.

Let $M\in G_r^{\rm arith}$. By Equation \eqref{eq:aritto}, we can write $M=D\circ A$, where $A\in {\rm Sp}(e,q^{d/e})$ and $D=R_{a^{q^i}}\circ\sigma \circ S$, with $\sigma\in{\rm Aut}(\mathbb{F}_{q^{d/e}}:\mathbb{F}_q)$ and $S\in\mathcal{Z}_q$. Therefore, $\eta(D)$ is a $d\times d$ diagonal-block matrix over $\mathbb{F}_q$ with $e$ blocks of size $d/e$, and the same holds for its inverse matrix $\eta(D)^{-1}$.

Let $\mathcal{B}$ be the $\mathbb{F}_q$-basis of $\mathbb{F}_{q^{d/e}}$ associated with $\eta=\eta|_{\mathcal{A,\mathcal{C}}}$ as in Section \ref{Sect:EmbeddingIssue}.
Consider the columns $c_1,\ldots,c_{e/2}\in\mathbb{F}_{q}^d$ of $\eta(D)^{-1}$ in positions 
$$1, 1+d/e, 1+2d/e,\ldots,  1+(e/2-1)d/e,$$
respectively.
For any $i=1,\ldots,e/2$, define the vector $f_i\in\mathbb{F}_{q^{d/e}}^e$ such that, for any $j=1,\ldots,e$, the $j$-th block of $c_i$ gives the coordinates over $\mathcal{B}$ of the $j$-th entry of $f_i$.

Note that the only non-zero entry of $f_i$ is in position $i$. By Equation \eqref{eq:symp}, this implies that $\omega(f_i,f_{i^\prime})=0$ for any $1\leq i,i^{\prime}\leq e/2$, where $\omega$ is the symplectic form associated with ${\rm Sp}(e,q^{d/e})$.
Then $\{f_1,\ldots,f_{e/2}\}$ can be extended to a symplectic basis of $\mathbb{F}_{q^{d/e}}^e$, and hence there exists $\widetilde{D}\in{\rm Sp}(e,q^{d/e})\leq G_r^{\rm geom}$ with $f_i$ as $i$-th column, $i=1,\ldots,e/2$.
As described in Section \ref{Sect:EmbeddingIssue} (see Equation \eqref{Def:varphi}), the first $e/2$ columns of $\eta(\widetilde{D})$ are respectively $c_1,\ldots,c_{e/2}$.

Finally, let $B:=A^{-1}\widetilde{D}\in G_r^{\rm geom}$. Then $MB=D\widetilde{D}$, and hence $\eta(MB)=\eta(D)\cdot\eta(\widetilde{D})$ has the same $(1+jd/e)$-th column of the identity matrix $\mathbb{I}_d$ for any $j=0,\ldots,e/2-1$.
Therefore
\[{\rm rank}(\eta(MB)-\mathbb{I}_d)\leq d-e/2<d-1\]
and the proof is complete.
\end{proof}

\section{Acknowledgements} 

This research was supported by the Italian National Group for Algebraic and Geometric Structures and their Applications (GNSAGA - INdAM).
The first author is funded by Progetto {\em Applicazioni di curve algebriche su campi finiti alla Teoria dei Codici e alla Crittografia}, Fondo Ricerca di Base, 2019, Università degli Studi di Perugia.
The second author is funded by Progetto {\em Strutture Geometriche, Combinatoria e loro Applicazioni}, Fondo Ricerca di
Base, 2019, Università degli Studi di Perugia.

\bibliographystyle{acm}
\bibliography{biblio.bib}

\end{document}